\documentclass[12pt]{article}
\usepackage{txfonts}

\def\T{\Gamma} \def\D{\Delta} \def\Th{\Theta}
\def\Ld{\Lambda} \def\E{\Sigma} \def\O{\Omega}
\def\a{\alpha} \def\b{\beta} \def\g{\gamma} \def\d{\delta} \def\e{\varepsilon}
\def\r{\rho} \def\o{\sigma} \def\t{\tau} \def\w{\omega} \def\k{\kappa}
\def\th{\theta} \def\ld{\lambda} \def\ph{\varphi} \def\z{\zeta}
\def\A{$A$~} \def\G{$G$~} \def\H{$H$~} \def\K{$K$~} \def\M{$M$~} \def\N{$N$~}
\def\P{$P$~} \def\Q{$Q$~} \def\R{$R$~} \def\V{$V$~} \def\X{$X$~} \def\Y{$Y$~}
\def\rmA{{\bf A}} \def\rmD{{\bf D}} \def\rmS{{\bf S}} \def\rmK{{\bf K}}
\def\rmM{{\bf M}} \def\Z{{\Bbb Z}} \def\GL{{\bf GL}} \def\C{Cayley}
\def\oa{\ovl A} \def\og{\ovl G} \def\oh{\ovl H} \def\ob{\ovl B} \def\oq{\ovl Q}
\def\oc{\ovl C} \def\ok{\ovl K} \def\ol{\ovl L} \def\om{\ovl M} \def\on{\ovl N}
\def\op{\ovl P} \def\oR{\ovl R} \def\os{\ovl S} \def\ot{\ovl T} \def\ou{\ovl U}
\def\ov{\ovl V} \def\ow{\ovl W} \def\ox{\ovl X} \def\oT{\ovl\T}
\def\lg{\langle} \def\rg{\rangle}
\def\di{\bigm|} \def\Di{\Bigm|} \def\nd{\mathrel{\bigm|\kern-.7em/}}
\def\Nd{\mathrel{\not\,\Bigm|}} \def\edi{\bigm|\bigm|}
\def\m{\medskip} \def\l{\noindent} \def\x{$\!\,$}  \def\J{$-\!\,$}
\def\Hom{\hbox{\rm Hom}} \def\Aut{\hbox{\rm Aut}} \def\Inn{\hbox{\rm Inn}}
\def\Syl{\hbox{\rm Syl}} \def\Sym{\hbox{\rm Sym}} \def\Alt{\hbox{\rm Alt}}
\def\Ker{\hbox{\rm Ker}} \def\fix{\hbox{\rm fix}} \def\mod{\hbox{\rm mod}}
\def\psl{{\bf P\!SL}} \def\Cay{\hbox{\rm Cay}} \def\Mult{\hbox{\rm Mult}}
\def\val{\hbox{\rm Val}} \def\Sab{\hbox{\rm Sab}} \def\supp{\hbox{\rm supp}}
\def\qed{\hfill $\Box$} \def\qqed{\qed\vspace{3truemm}}
\def\CS{\Cay(G,S)} \def\CT{\Cay(G,T)}
\def\h{\heiti\bf} \def\hs{\ziti{E}\bf} \def\st{\songti} \def\ft{\fangsong}
\def\kt{\kaishu} \def\heit{\hs\relax} \def\songt{\st\rm\relax}
\def\fangs{\ft\rm\relax} \def\kaish{\kt\rm\relax} \def\fs{\footnotesize}
\begin{document}

\newtheorem{theorem}{Theorem}[section]
\newtheorem{corollary}[theorem]{Corollary}
\newtheorem{definition}[theorem]{Definition}
\newtheorem{conjecture}[theorem]{Conjecture}
\newtheorem{question}[theorem]{Question}
\newtheorem{lemma}[theorem]{Lemma}
\newtheorem{proposition}[theorem]{Proposition}
\newtheorem{example}[theorem]{Example}
\newtheorem{problem}[theorem]{Problem}
\newenvironment{proof}{\noindent {\bf
Proof.}}{\rule{3mm}{3mm}\par\medskip}
\newcommand{\remark}{\medskip\par\noindent {\bf Remark.~~}}
\newcommand{\pp}{{\it p.}}
\newcommand{\de}{\em}

\def\dfrac{\displaystyle\frac} \def\ovl{\overline}
\def\for{\forall~} \def\exi{\exists~} \def\c{\subseteq}
\def\iif{\Longleftrightarrow} \def\Rto{\Rightarrow} \def\Lto{\Leftarrow}
\def\T{\Gamma} \def\D{\Delta} \def\Th{\Theta}
\def\Ld{\Lambda} \def\E{\Sigma} \def\O{\Omega}
\def\a{\alpha} \def\b{\beta} \def\g{\gamma} \def\d{\delta} \def\e{\varepsilon}
\def\r{\rho} \def\o{\sigma} \def\t{\tau} \def\w{\omega} \def\k{\kappa}
\def\th{\theta} \def\ld{\lambda} \def\ph{\varphi} \def\z{\zeta}
\def\A{$A$~} \def\G{$G$~} \def\H{$H$~} \def\K{$K$~} \def\M{$M$~} \def\N{$N$~}
\def\P{$P$~} \def\Q{$Q$~} \def\R{$R$~} \def\V{$V$~} \def\X{$X$~} \def\Y{$Y$~}
\def\rmA{{\bf A}} \def\rmD{{\bf D}} \def\rmS{{\bf S}} \def\rmK{{\bf K}}
\def\rmM{{\bf M}} \def\Z{{\Bbb Z}} \def\GL{{\bf GL}} \def\C{Cayley}
\def\oa{\ovl A} \def\og{\ovl G} \def\oh{\ovl H} \def\ob{\ovl B} \def\oq{\ovl Q}
\def\oc{\ovl C} \def\ok{\ovl K} \def\ol{\ovl L} \def\om{\ovl M} \def\on{\ovl N}
\def\op{\ovl P} \def\oR{\ovl R} \def\os{\ovl S} \def\ot{\ovl T} \def\ou{\ovl U}
\def\ov{\ovl V} \def\ow{\ovl W} \def\ox{\ovl X} \def\oT{\ovl\T}
\def\lg{\langle} \def\rg{\rangle}
\def\di{\bigm|} \def\Di{\Bigm|} \def\nd{\mathrel{\bigm|\kern-.7em/}}
\def\Nd{\mathrel{\not\,\Bigm|}} \def\edi{\bigm|\bigm|}
\def\m{\medskip} \def\l{\noindent} \def\x{$\!\,$}  \def\J{$-\!\,$}
\def\Hom{\hbox{\rm Hom}} \def\Aut{\hbox{\rm Aut}} \def\Inn{\hbox{\rm Inn}}
\def\Syl{\hbox{\rm Syl}} \def\Sym{\hbox{\rm Sym}} \def\Alt{\hbox{\rm Alt}}
\def\Ker{\hbox{\rm Ker}} \def\fix{\hbox{\rm fix}} \def\mod{\hbox{\rm mod}}
\def\psl{{\bf P\!SL}} \def\Cay{\hbox{\rm Cay}} \def\Mult{\hbox{\rm Mult}}
\def\val{\hbox{\rm Val}} \def\Sab{\hbox{\rm Sab}} \def\supp{\hbox{\rm supp}}
\def\qed{\hfill $\Box$} \def\qqed{\qed\vspace{3truemm}}
\def\CS{\Cay(G,S)} \def\CT{\Cay(G,T)}
\def\h{\heiti\bf} \def\hs{\ziti{E}\bf} \def\st{\songti} \def\ft{\fangsong}
\def\kt{\kaishu} \def\heit{\hs\relax} \def\songt{\st\rm\relax}
\def\fangs{\ft\rm\relax} \def\kaish{\kt\rm\relax} \def\fs{\footnotesize}

\newcommand{\JEC}{{\it Europ. J. Combinatorics},  }
\newcommand{\JCTB}{{\it J. Combin. Theory Ser. B.}, }
\newcommand{\JCT}{{\it J. Combin. Theory}, }
\newcommand{\JGT}{{\it J. Graph Theory}, }
\newcommand{\ComHung}{{\it Combinatorica}, }
\newcommand{\DM}{{\it Discrete Math.}, }
\newcommand{\ARS}{{\it Ars Combin.}, }
\newcommand{\SIAMDM}{{\it SIAM J. Discrete Math.}, }
\newcommand{\SIAMADM}{{\it SIAM J. Algebraic Discrete Methods}, }
\newcommand{\SIAMC}{{\it SIAM J. Comput.}, }
\newcommand{\ConAMS}{{\it Contemp. Math. AMS}, }
\newcommand{\TransAMS}{{\it Trans. Amer. Math. Soc.}, }
\newcommand{\AnDM}{{\it Ann. Discrete Math.}, }
\newcommand{\NBS}{{\it J. Res. Nat. Bur. Standards} {\rm B}, }
\newcommand{\ConNum}{{\it Congr. Numer.}, }
\newcommand{\CJM}{{\it Canad. J. Math.}, }
\newcommand{\JLMS}{{\it J. London Math. Soc.}, }
\newcommand{\PLMS}{{\it Proc. London Math. Soc.}, }
\newcommand{\PAMS}{{\it Proc. Amer. Math. Soc.}, }
\newcommand{\JCMCC}{{\it J. Combin. Math. Combin. Comput.}, }
\newcommand{\GC}{{\it Graphs Combin.}, }

\title{ Maximum-Size Independent Sets and Automorphism Groups of Tensor Powers of the Even Derangement Graphs \thanks{
This work is supported by National Natural Science Foundation of
China (No:10971137), the National Basic Research Program (973) of
China (No.2006CB805900),  and a grant of Science and Technology
Commission of Shanghai Municipality (STCSM, No: 09XD1402500)
.\newline \indent
  $^{\dagger}$Correspondent author: Xiao-Dong Zhang
 (Email: xiaodong@sjtu.edu.cn)}}
\author{
 Yun-Ping Deng$^{1}$,  Fu-Ji Xie$^2$ and  Xiao-Dong Zhang$^{1,\dagger}$\\
{\small $^1$Department of Mathematics, Shanghai Jiao Tong University, Shanghai 200240, P.R.China}\\
{\small $^2$ Antai College of Economics \& Management Shanghai Jiao
Tong University, Shanghai
200052, P.R.China}\\
{\small Emails: dyp612@hotmail.com, xiefuji@sjtu.edu.cn, xiaodong@sjtu.edu.cn}\\
 }
\date{}
\maketitle
 \begin{abstract}
   Let $A_n$ be the alternating group of even permutations of $X:=\{1,2,\cdots,n\}$ and ${\mathcal E}_n$ the set of
   even derangements on $X.$ Denote by $A\T_n^q$ the tensor product of $q$ copies of $A\T_n,$
   where the Cayley graph $A\T_n:=\T(A_n,{\mathcal E}_n)$ is called the even derangement graph. In this paper, we intensively investigate
   the properties of $A\T_n^q$ including connectedness, diameter, independence number, clique number, chromatic number and
   the maximum-size independent sets of $A\T_n^q.$ By using the result on the maximum-size independent sets $A\T_n^q$, we completely
   determine the full automorphism groups of $A\T_n^q.$

 \end{abstract}

{{\bf Key words:} Automorphism group; Cayley graph; tensor product; maximum-size independent sets; alternating group. }\\

{{\bf AMS Classifications:} 05C25, 05C69}
\vskip 0.5cm

\section{Introduction}
For a simple graph $\T,$ we use $V(\T), E(\T)$ and $\Aut(\T)$ to
denote its vertex set, edge set and full automorphism group,
respectively. We denote by $N_{\T}(v)$ the neighbourhood of a vertex
$v$ in $\T.$ Let \G be a finite group and $S$ a subset of \G not
containing the identity element $1$ with $S=S^{-1}.$ The {\it Cayley
graph} $\T:=\T(G,S)$ on \G with respect to $S$ is defined by
$$V(\T){=}G,~E(\T){=}\{(g,sg): g{\in} G,\ s{\in} S\}.$$ Clearly, $\T$
is a $|S|$-regular and vertex-transitive graph, since $\Aut(\T)$
contains the right regular representation $R(G)$ of \G. Moreover,
$\T$ is connected if and only if \G is generated by $S.$

Let $S_n$ be the {\it symmetric group} and $A_n$ the {\it
alternating group} on $X=\{1,2,\cdots,n\}.$ Let ${\mathcal
D}_n:=\{\o\in S_n:x^{\,\sigma}\neq x,\forall x\in X\}$ and
${\mathcal E}_n:={\mathcal D}_n\cap A_n$ denote the degrangements
and the even derangements on \X respectively. Then the graph
$\T_n:=\T(S_n,{\mathcal D}_n)$ and $A\T_n:=\T(A_n,{\mathcal E}_n)$
are called the {\it derangement graph} \cite {Paul} and the {\it
even derangement graph} on \X respectively.

The {\it tensor product} $\T_1\otimes\T_2$ of two graphs $\T_1$ and
$\T_2$ is the graph with vertex set $V(\T_1)\times V(\T_2)$ and edge
set consisting of those pairs of vertices $(u_1,u_2),\,(v_1,v_2)$
where $u_1$ is adjacent to $v_1$ in $\T_1$ and $u_2$ is adjacent to
$v_2$ in $\T_2$. A {\it projection} is a homomorphism
$pr_{i,n}:\T^q\rightarrow\T$ given by
$pr_{i,n}(x_1,x_2,\cdots,x_q)=x_i$ for some $i,$ where $\T^q$ is the
tensor product of $q$ copies of a graph $\T.$ By the definition of
tensor product, it is easy to see that $A\T_n^q$ is the Cayley graph
$\T(A_n^q,{\mathcal E}_n^q),$ where $A_n^q$ is the direct product of
$q$ copies of $A_n$ and ${\mathcal
E}_n^q:=\{(\sigma_1,\sigma_2,\cdots,\sigma_q):\sigma_i\in{\mathcal
E}_n,i=1,2,\cdots,q\}.$

A family $I\subseteq S_n$ is {\it intersecting} if any two elements
have at least one common entry. It is easy to see that an
intersecting family of maximal size in $S_n$ corresponds to a
maximum-size independent set in $\T_n.$ In \cite {Cameron}, Cameron
and Ku showed that the only intersecting families of maximal size in
$S_n$ are the cosets of point stabilizers. In \cite {Ku}, Ku and
Wong proved that analogous results hold for the alternating group
and the direct product of symmetric groups, which equivalently shows
that the structure of maximum-size independent sets of $A\T_n$ is as
follows:

\begin{proposition} (Theorem 1.2 in ${{\fs\cite{Ku}}}$)\label{pr-1.1}
All the maximum-size independent sets of $A\T_n~(n\geq5)$ are
$B_{i,j}=\{\sigma\in A_n:\,i^{\sigma}=j\},\,i,j=1,2,\cdots,n.$ In
particular, each $|B_{i,j}|=\frac{(n-1)!}{2}.$
\end{proposition}

In this paper, we prove that the result analogous to \cite{Cameron}
holds for the direct product of the alternating groups, which can be
equivalently stated as follows:

\begin{theorem} \label{dl-1.2}
All the maximum-size independent sets of $A\T_n^q~(q\geq1,n\geq5)$
are
$$B_{i,j}^{(k)}=\{(\sigma_1,\sigma_2,\cdots,\sigma_q)\in A_n^q:
\,i^{\sigma_k}=j\},\,i,j=1,2,\cdots,n;\,k=1,2\cdots,q.$$ In
particular, the independence number of $A\T_n^q$ is
$$|B_{i,j}^{(k)}|=\frac{(n-1)!n!^{q-1}}{2^q}.$$
\end{theorem}

{\bf Remark.} Generally speaking, for a graph $\T,$ all maximum-size
independent sets of $\T^q$ are not necessarily preimages of
maximum-size independent sets of $\T$ under projections (see
\cite{Ku2,Larose2}). Theorem \ref{dl-1.2} shows that all
maximum-size independent sets of $A\T_n^q$ are preimages of
maximum-size independent sets of $A\T_n$ under projections.

Many researchers (see \cite
{Cameron,deza1977,Eggleton1985,Larose,Paul,Rasmussen}) have studied
the properties of $\T_n,$ such as the clique number, the chromatic
number, the independence number, maximum-size independent sets and
so on. Motivated by the nice structures of $\T_n,$ here we show that
$A\T_n^q$ have the similar nice structures. For example, we obtain
that the diameter $D(A\T_n^q)=2,$ the clique number
$\omega(A\T_n^q)=n$ and the chromatic number $\chi(A\T_n^q)=n.$

Cayley graphs are of general interest in the field of Algebraic
Graph Theory due to their good properties, especially their high
symmetry. One difficult problem in Algebraic Graph Theory is to
determine the automorphism groups of Cayley graphs. Although there
are some nice results on the automorphism groups of Cayley graphs
(see \cite {Fang,Feng2,Feng1,Godsil2,Huan,Xu,Z,Zhang}), we still
lack enough understanding on them. In this paper, we completely
determine the automorphism groups of $A\T_n^q,$ which in fact gives
a kind of method on the computation of automorphism group of Cayley
graph by using the characterization of the maximum-size independent
sets. Another main result of this paper is as follows:

\begin{theorem} \label{dl-1.3}
Define the mapping $\varphi_k:\,A_n^q\rightarrow A_n^q$ as
$(\sigma_1,\cdots,\sigma_{k-1},\sigma_k,\sigma_{k+1},\cdots,\sigma_q)^{\varphi_k}=(\sigma_1,\cdots,\sigma_{k-1},\sigma_k^{-1},\sigma_{k+1},\cdots,\sigma_q)$
for $k=1,2,\cdots,q.$ For $q\geq1$ and $n\geq 5,$
$$\Aut(A\T_n^q)=(R(A_n^q)\rtimes (\Inn(S_n)\wr S_q))\rtimes Z_2^q,$$ where $\Inn(S_n)\,(\,\cong S_n)$ is the inner automorphism group of
$S_n,$
$Z_2^q=\langle\varphi_1\rangle\times\langle\varphi_2\rangle\times\cdots\langle\varphi_q\rangle$
and $\,\Inn(S_n)\wr S_q$ denotes the wreath product of $\,\Inn(S_n)$
and $S_q.$
\end{theorem}

{\bf Remark.} Sanders and George \cite {Sanders} showed that for a
graph $\T,$ $\Aut(\T^2)\geq\Aut(\T)\wr S_2,$ where $\wr$ denotes the
wreath product, however, the equality cannot hold in most
situations. Theorem \ref{dl-1.3} implies that
$\Aut(A\T_n^q)=\Aut(A\T_n)\wr S_q.$

The rest part of this paper is organized as follows. In Section 2,
we give the connectedness and diameter of $A\T_n^q.$ In Section 3,
we determine the independence number and the structure of
maximum-size independent sets of $A\T_n^q,$ as its corollary, we
obtain the clique number and chromatic number of $A\T_n^q.$ In
section 4, we completely determine the full automorphism groups of
$A\T_n^q.$

\section{The connectedness and diameter}

In this section, we give the connectedness and diameter of
$A\T_n^q.$

For a group $G,$ we denote the automorphism group and the inner
automorphism group of $G$ by $\Aut(G)$ and $\Inn(G),$ respectively.
Next we need the following known result:

\begin{proposition} ${{\fs\cite{Suzuki}\,[III,\,(2.17)-(2.20)]}}$ \label{pr-2.1}
If $n\geq2$ and $n\neq6,$ then $\Aut(A_n)=\Inn(S_n).$ If $n=6,$ then
$|\Aut(A_6):\Inn(S_6)|=2,$ and for each $\alpha\in
\Aut(A_6){\setminus}\Inn(S_6),$ $\alpha$ maps a $3$-cycle to a
product of two disjoint $3$-cycles.
\end{proposition}

\begin{lemma} \label {yl-2.2}
If $n\geq 5,$ then the even derangement graph $A\T_n$ is connected.
\end{lemma}

\begin{proof}
By Theorem 2.8 of page 293 in \cite{Suzuki}, the alternating group
$A_n~(n\geq 5)$ is generated by the totality of $3$-cycles. Clearly
$(1\,2\,3)=(1\,2\,\cdots\,n)^2\cdot(n\,n-1\,\cdots\,1)^2(1\,2\,3)$
and $(1\,2\,\cdots\,n)^2,\,(n\,n-1\,\cdots\,1)^2(1\,2\,3)\in
{\mathcal E}_n$ by $n\geq 5.$

For any $3$-cycle $(i\,j\,k),$ there exists a $\phi\in \Inn(S_n)$
such that $(1\,2\,3)^{\phi}=(i\,j\,k).$ By Proposition~\ref
{pr-2.1}, we have $\Aut(A_n,{\mathcal E}_n)=\{\phi\in
\Aut(A_n):\,{\mathcal E}_n^{\,\phi}={\mathcal E}_n\}=\Inn(S_n).$
Thus

$$(i\,j\,k)=(1\,2\,3)^{\phi}=[(1\,2\,\cdots\,n)^2]^{\phi}\cdot[(n\,n-1\,\cdots\,1)^2(1\,2\,3)]^{\phi}$$
and
$$[(1\,2\,\cdots\,n)^2]^{\phi},\,[(n\,n-1\,\cdots\,1)^2(1\,2\,3)]^{\phi}
\in {\mathcal E}_n.$$ So the alternating group $A_n~(n\geq 5)$ is
generated by ${\mathcal E}_n,$ which implies that $A\T_n$ is
connected.\end{proof}

{\bf Remark.} If $n=3,$ clearly $A_3=\lg {\mathcal E}_3\rg,$ so
$A\T_3$ is connected. If $n=4,$ then $A_4\neq\lg {\mathcal
E}_4\rg=\{1,\,(1\,2)(3\,4),\,(1\,3)(2\,4),\,(1\,4)(2\,3)\},$ so
$A\T_4$ is not connected.

\begin{lemma} ${}^{{\fs\cite{Graham}}}$\label{yl-2.3}
(i) The tensor product of two connected graphs is bipartite if and
only if at least one of them is bipartite.

(ii) The tensor product of two connected graphs is disconnected if
and only if both factors are bipartite.
\end{lemma}

\begin{theorem} \label{dl-2.4}
$A\T_n^q$ is connected and non-bipartite for any $q\geq 1$ and
$n\geq 5.$
\end{theorem}

\begin{proof}
Since $A\T_n^q=\underbrace{A\T_n\otimes\cdots\otimes A\T_n}_q$ and
$A\T_n$ is connected and non-bipartite for $n\geq 5$ by Lemma
\ref{yl-2.2}, the assertion holds by Lemma \ref{yl-2.3}.
\end{proof}

\begin{lemma} \label {yl-2.5}
For any $g_1,g_2\in A_n~(n\geq 5),$ there exists a $g\in A_n$ such
that $g\in N_{A\T_n}(g_1)\cap N_{A\T_n}(g_2).$
\end{lemma}

\begin{proof} If $n=5,$ we have
$$(a_1,a_2,a_3,a_4,a_5)=(a_1,a_4,a_2,a_5,a_3)^2,$$
$$(a_1,a_2)(a_3,a_4)=(a_5,a_1,a_3,a_2,a_4)(a_1,a_5,a_3,a_2,a_4),$$
$$(a_1,a_2,a_3)=(a_1,a_5,a_3,a_4,a_2)(a_1,a_3,a_5,a_2,a_4),$$
$$1=(a_1,a_2,a_3,a_4,a_5)(a_5,a_4,a_3,a_2,a_1),$$
that is, for any $x\in A_5,$ there exist $s_1,s_2\in {\mathcal E}_5$
such that $x=s_1s_2.$ Now for $x=g_1g_2^{-1},$ we have
$g_1g_2^{-1}=s_1s_2,\,s_1,s_2\in {\mathcal E}_5,$ i.e.
$g_1=s_1s_2g_2.$ Set $g:=s_2g_2.$ Clearly $g\in N_{A\T_5}(g_1)\cap
N_{A\T_5}(g_2).$

If $n\geq 6,$ by proposition 6 in \cite {Cameron}, for any
$g_1,g_2\in A_n,$ there exists a $g\in S_n$ such that $g\in
N_{\T_n}(g_1)\cap N_{\T_n}(g_2).$ That is, there exist $s_1,\,s_2\in
{\mathcal D}_n$ such that $g=s_1g_1=s_2g_2.$ If $g\in A_n,$ then
$s_1,\,s_2\in {\mathcal E}_n,$ so $g\in N_{A\T_n}(g_1)\cap
N_{A\T_n}(g_2)$ and the assertion holds. If $g\in S_n\setminus A_n,$
then $s_1,\,s_2\in {\mathcal D}_n\setminus {\mathcal E}_n.$ For any
$i\in X=\{1,2,\cdots,n\},$ select a
$j\in\{i,i^{s_1},i^{s_2},i^{s_1^{-1}},i^{s_2^{-1}}\}\neq
\emptyset~(n\geq 6).$ Set
$$g':=(i\,j)g,\,s_1':=(i\,j)s_1,\,s_2':=(i\,j)s_2.$$
Thus $g'=s_1'g_1=s_2'g_2$ and $s_1',s_2'\in {\mathcal E}_n$ by $j\in
X\setminus\{i,i^{s_1},i^{s_2},i^{s_1^{-1}},i^{s_2^{-1}}\}.$ Hence
$g'\in N_{A\T_n}(g_1)\cap N_{A\T_n}(g_2)$ and the assertion
holds.\end{proof}

\begin{theorem} \label {dl-2.6}
If $n\geq5,$ then $diam(A\T_n^q)=2,$ where $diam(A\T_n^q)$ is the
diameter of $A\T_n^q.$
\end{theorem}

\begin{proof} For any
$(\sigma_1,\sigma_2,\cdots,\sigma_q),(\tau_1,\tau_2,\cdots,\tau_q)\in
A_n^q,$ by Lemma \ref{yl-2.5}, there exist $\varsigma_i\in
A_n~(i=1,2,\cdots,q)$ such that $\varsigma_i\in
N_{A\T_n}(\sigma_i)\cap N_{A\T_n}(\tau_i).$ So there exists a
$(\varsigma_1,\varsigma_2,\cdots,\varsigma_q)\in A_n^q$ such that
$$(\varsigma_1,\varsigma_2,\cdots,\varsigma_q)\in
N_{A\T_n^q}((\sigma_1,\sigma_2,\cdots,\sigma_q))\cap
N_{A\T_n^q}((\tau_1,\tau_2,\cdots,\tau_q)),$$ which implies that any
two vertices in $A\T_n^q$ have at least a common neighbourhood.
Hence $diam(A\T_n^q)=2.$
\end{proof}

\section{The stucture of maximum-size independent sets}

In this section we characterize the structure of maximum-size
independent sets of $A\T_n^q$ for $q\geq 1,$ which is a
generalization of Theorem 1.2 in \cite {Ku}. First we give the
independence number of $A\T_n^q$ as follows:

\begin{lemma} \label{yl-3.1}
For any $q\geq 1,\,n\geq 5,$ the independence number of $A\T_n^q$ is
given by  $$\alpha(A\T_n^q)=\frac{(n-1)!n!^{q-1}}{2^q}.$$
\end{lemma}

\begin{proof}
By Proposition 1.3 in \cite{Alon}, we have

$$\frac{\alpha(A\T_n^q)}{|A_n^q|}=\frac{\alpha(A\T_n)}{|A_n|}
\Rightarrow\alpha(A\T_n^q)=\frac{\alpha(A\T_n)\cdot|A_n^q|}{|A_n|}.$$
Then by Proposition \ref{pr-1.1}, we obtain
$$\alpha(A\T_n^q)=\frac{\frac{(n-1)!}{2}\cdot(\frac{n!}{2})^q}{\frac{n!}{2}}
=\frac{(n-1)!n!^{q-1}}{2^q}.$$
Thus the assertion holds.\end{proof}

For any two graphs $H_1$ and $H_2,$ a map $\phi$ from $V(H_1)$ to
$V(H_2)$ is {\it homomorphism} if $\{u^{\phi},v^{\phi}\}\in E(H_2)$
whenever $\{u,v\}\in E(H_1),$ i.e. $\phi$ is a edge-preserving map.
Next we need the following fundamental result of Albertson and
Collins \cite{Albertson} which is also called 'No-Homomorphism
Lemma'.

\begin{lemma}  ${{\fs\cite{Albertson}}}$\label{yl-3.2}
Let $H_1$ and $H_2$ be graphs such that $H_2$ is vertex transitive
and there exists a homomorphism $\phi:~V(H_1)\rightarrow V(H_2).$
Then
\begin{equation}\label{2}
\frac{\alpha(H_1)}{|V(H_1)|}\geq\frac{\alpha(H_2)}{|V(H_2)|}
\end{equation}
Furthermore, if equality holds in (\ref{2}), then for any
independent set $I$ of cardinality $\alpha(H_2)$ in $H_2$,
$I^{{\phi}^{-1}}$ is an independent set of cardinality $\alpha(H_1)$
in $H_1.$
\end{lemma}

\begin{lemma} \label{yl-3.3}
Let $H=(V_1,V_2,E)$ be a $d$-regular bipartite graph whose partition
has the parts $V_1$ and $V_2$ with $|V_1|=|V_2|.$ If $H$ is
connected, then $|S|<|N_H(S)|$ for any $S\subsetneq V_1,$ where
$N_H(S)$ is the neighborhood of $S$ in $H.$
\end{lemma}

\begin{proof}
Let $T=N_H(S)$ and $E(S,T)=\{(s,t)\in E:s\in S,t\in T\}.$ Then
$$d|S|=|E(S,T)|\leq |E(V_1,T)|=d|T|.$$ If
$|S|=|T|,$ then $E(S,T)|=|E(V_1,T)|,$ i.e. any vertex $u\in S\cup T$
is not adjacent to any vertex $v\not\in S\cup T,$ which contradicts
the connectedness of $H.$ Thus $|S|<|T|=|N_H(S)|.$
\end{proof}

\begin{lemma} \label{yl-3.4}
All the maximum-size independent sets of $A\T_n^2~(n\geq7)$ are
$$B_{i,j}^{(k)}=\{(g_1,g_2)\in A_n^2:
\,i^{g_k}=j\},\,i,j=1,2,\cdots,n;\,k=1,2.$$
\end{lemma}

\begin{proof}
Set ${\mathcal B}=\{B_{i,j}^{(k)}:i,j=1,2,\cdots,n;\,k=1,2\}.$
Clearly $|B_{i,j}^{(k)}|=\frac{(n-1)!n!}{4},$ which is equal to
$\alpha(A\T_n^2)$ by Lemma \ref{yl-3.1}. That is, $B_{i,j}^{(k)}$ is
a maximum-size independent set of $A\T_n^2.$ Next for any
maximum-size independent set $I$ of $A\T_n^2,$ it suffices to show
that $I\in{\mathcal B}.$

Define a homomorphism $\phi$ from $A\T_n$ to $A\T_n^2$ as
$g^{\phi}=(g,g).$ Without loss of generality, we may assume that the
identity $Id=(id,id)\in I.$ By Proposition \ref{pr-1.1}, Lemmas
\ref{yl-3.1} and \ref{yl-3.2}, $I^{{\phi}^{-1}}$ is a maximum
independent set of $A\T_n.$ So $I^{{\phi}^{-1}}=\{g\in
A_n:~i_0^g=j_0\}$ for some $i_0,\,j_0$ by Proposition \ref{pr-1.1}.
Since $id=(Id)^{{\phi}^{-1}}\in I^{\phi^{-1}},$
$I^{\phi^{-1}}=\{g\in A_n:~i_0^g=i_0\}$ for some $i_0.$ Therefore
$I\supseteq I_0:=(I^{{\phi}^{-1}})^{\phi}=\{(g,g)\in
A_n^2:~i_0^g=i_0\}.$ Next we shall show that $I\in{\mathcal B}$ by the following four Claims:\\

{\bf Claim 1.} For any $(g_1,g_2)\in I,$ either $i_0^{g_1}=i_0$ or
$i_0^{g_2}=i_0.$

Suppose on the contrary that $i_0^{g_1}\neq i_0$ and $i_0^{g_2}\neq
i_0.$ By Lemma \ref{yl-2.5}, there exists a $g\in A_n$ such that
$g\in N_{A\T_n}(g_1)\cap N_{A\T_n}(g_2).$ That is, there exist
$s_1,s_2\in {\mathcal E}_n$ such that $g=s_1g_1=s_2g_2.$

If $i_0^g=i_0,$ then $(g,g)\in I_0\subseteq I$ and
$\{(g,g),(g_1,g_2)\}\in E(A\T_n^2),$ which contradicts the fact that
$(g_1,g_2)\in I$ and $I$ is an independent set.

If $i_0^g\neq i_0,$ select a $j\in X\setminus
\{i_0,i_0^{g^{-1}},i_0^{s_1},i_0^{s_2},
i_0^{g^{-1}s_1^{-1}},i_0^{g^{-1}s_2^{-1}}\}\neq\emptyset~(n\geq 7).$
Set
$$g'=(i_0,i_0^{g^{-1}},j)g,\,s_1'=(i_0,i_0^{g^{-1}},j)s_1,\,s_2'=(i_0,i_0^{g^{-1}},j)s_2.$$
Thus $i_0^{g'}= i_0,\,g'=s_1'g_1=s_2'g_2$ and $s_1',s_2'\in
{\mathcal E}_n$ by $j\in X\setminus
\{i_0,i_0^{g^{-1}},i_0^{s_1},i_0^{s_2},
i_0^{g^{-1}s_1^{-1}},i_0^{g^{-1}s_2^{-1}}\}.$ So $(g',g')\in
I_0\subseteq I$ and$\{(g',g'),(g_1,g_2)\}\in E(A\T_n^2),$ which as
above yields a contradiction.

Hence Claim 1 holds.\\

Set $$J_0=\{(g_1,g_2)\in A_n^2:i_0^{g_1}=i_0~and~i_0^{g_2}=i_0\},$$
    $$J_1=\{(g_1,g_2)\in A_n^2:i_0^{g_1}=i_0~and~i_0^{g_2}\neq i_0\},$$
    $$J_2=\{(g_1,g_2)\in A_n^2:i_0^{g_1}\neq i_0~and~i_0^{g_2}=i_0\}.$$

Clearly $|J_0|=\frac{(n-1)!^2}{4},|J_1|=|J_2|=\frac{(n-1)(n-1)!^2}{4}.$\\

{\bf Claim 2.} $A\T_n^2[J_1\cup J_2]$ is connected, where
$A\T_n^2[J_1\cup J_2]$ denote the induced subgraph of $A\T_n^2$ by
$J_1\cup J_2.$

For any $(\sigma_1,\sigma_2),(\tau_1,\tau_2)\in J_1,$ clearly they
are not adjacent in $A\T_n^2.$ By Theorem \ref{dl-2.6}, there exists
a $(\varsigma_1,\varsigma_2)\in A_n^2$ such that
$\{(\sigma_1,\sigma_2),(\varsigma_1,\varsigma_2)\}$ and
$\{(\tau_1,\tau_2),(\varsigma_1,\varsigma_2)\}\in E(A\T_n^2).$ That
is, there exist $s_1,s_2,t_1,t_2\in {\mathcal E}_n$ such that
$\varsigma_1=s_1\sigma_1=t_1\tau_1,\,\varsigma_2=s_2\sigma_2=t_2\tau_2.$
Clearly $i_0^{\varsigma_1}=i_0^{s_1\sigma_1}\neq i_0.$

If $i_0^{\varsigma_2}=i_0,$ then $(\varsigma_1,\varsigma_2)\in J_2.$

If $i_0^{\varsigma_2}\neq i_0,$ then select a $j\in X\setminus
\{i_0,i_0^{\varsigma_2^{-1}},i_0^{s_2},i_0^{t_2},
i_0^{\varsigma_2^{-1}s_2^{-1}},i_0^{\varsigma_2^{-1}t_2^{-1}}\}\neq\emptyset~(n\geq
7).$ Set
$$\varsigma_2'=(i_0,i_0^{\varsigma_2^{-1}},j)\varsigma_2,\,s_2'=(i_0,i_0^{\varsigma_2^{-1}},j)s_2,\,t_2'=(i_0,i_0^{\varsigma_2^{-1}},j)t_2.$$
Thus and $i_0^{\varsigma_2'}=
i_0,\,\varsigma_2'=s_2'\sigma_2=t_2'\tau_2$ and $s_2',t_2'\in
{\mathcal E}_n$ by $j\in X\setminus
\{i_0,i_0^{\varsigma_2^{-1}},i_0^{s_2},i_0^{t_2},
i_0^{\varsigma_2^{-1}s_2^{-1}},i_0^{\varsigma_2^{-1}t_2^{-1}}\}.$ So
$(\varsigma_1,\varsigma_2')\in J_2$ and
$\{(\sigma_1,\sigma_2),(\varsigma_1,\varsigma_2')\},\{(\tau_1,\tau_2),(\varsigma_1,\varsigma_2')\}\in
E(A\T_n^2[J_1\cup J_2]).$

Similarly, for any $(\sigma_1,\sigma_2),(\tau_1,\tau_2)\in J_2,$
their exists $(\varsigma_1,\varsigma_2)\in J_1$ such that
$\{(\sigma_1,\sigma_2),(\varsigma_1,\varsigma_2)\},\\
\{(\tau_1,\tau_2),(\varsigma_1,\varsigma_2)\}\in
E(A\T_n^2[J_1\cup J_2]).$

Hence Claim 2 holds.\\

{\bf Claim 3.} Either $I\cap J_1=\emptyset$ or $I\cap
J_2=\emptyset.$

Suppose on the contrary that $I\cap J_1\neq\emptyset$ and $I\cap
J_2\neq\emptyset,$ consider the following two possible cases:

{\bf Case 1.} $I\cap J_1=J_1$ or $I\cap J_2=J_2.$

Since $I\cap(J_1\cup J_2)$ is an independent set, this case cannot
happen.

{\bf Case 2.} $I\cap J_1\subsetneq J_1$ and $I\cap J_2\subsetneq
J_2.$

It is easy to see that $A\T_n^2[J_1\cup J_2]$ is a regular bipartite
graph whose partition has the parts $J_1$ and $J_2$ with
$|J_1|=|J_2|.$ By Claim 2 and Lemma \ref{yl-3.3}, we have
$$|I\cap J_1|<|N_{A\T_n^2[J_1\cup J_2]}(I\cap J_1)|.$$
Since $I\cap(J_1\cup J_2)$ is an independent set, we have
\begin{eqnarray*}
&&I\cap J_2\subseteq J_2\setminus N_{A\T_n^2[J_1\cup J_2]}(I\cap J_1)\\
&\Rightarrow&|N_{A\T_n^2[J_1\cup J_2]}(I\cap J_1)|+|I\cap J_2|\leq
|J_2|\\
&\Rightarrow&|I\cap J_1|+|I\cap J_2|<|J_2|
\end{eqnarray*}
By Claim 1, $I=\bigcup_{i=0}^2(I\cap J_i).$ Since $J_i~(i=0,1,2)$
are pairwise disjoint, we have
\begin{eqnarray*}
|I|&=&|I\cap J_0|+|I\cap J_1|+|I\cap J_2|\\
&<&|J_0|+|J_2|\\
&=&\frac{(n-1)!^2}{4}+\frac{(n-1)(n-1)!^2}{4}\\
&=&\frac{(n-1)!n!}{4}
\end{eqnarray*}
which is a contradiction, since $|I|=\frac{(n-1)!n!}{4}$ by Lemma
\ref{yl-3.1}. Hence Claim 3 holds.\\

{\bf Claim 4.} Either $I=J_0\cup J_1$ or $I=J_0\cup J_2.$

By Claim 3, either $I\cap J_1=\emptyset$ or $I\cap J_2=\emptyset.$
If $I\cap J_1=\emptyset,$ then we have
\begin{eqnarray*}
&&\frac{(n-1)!n!}{4}=|I|=|I\cap J_0|+|I\cap J_2|\leq |J_0|+|J_2|=\frac{(n-1)!n!}{4}\\
&\Rightarrow&I\cap J_0=J_0,\,I\cap J_2=J_2.
\end{eqnarray*}

So $I=\bigcup_{i=0}^2(I\cap J_i)=J_0\cup J_2.$

Similarly, if $I\cap J_2=\emptyset,$ then $I=J_0\cup J_1.$ Hence
Claim 4 holds.\\

By Claim 4, we have $I\in B,$ which conclude the proof.
\end{proof}

\begin{lemma} ${{\fs\cite{Alon}}}$\label{yl-3.5}
Let $\T$ be a connected d-regular graph on n vertices and let
$d=\mu_1\geq\mu_2\geq\cdots\mu_n$ be the eigenvalues of the
adjacency matrix of $\T.$ If
$$\frac{\alpha(\T)}{n}=\frac{-\mu_n}{d-\mu_n},$$ then for every
integer $q\geq1,$
$$\frac{\alpha(\T^q)}{n^q}=\frac{-\mu_n}{d-\mu_n}.$$
Moreover, if $\T$ is also non-bipartite, and if $I$ is an
independent set of size $\frac{-\mu_n}{d-\mu_n}n^q$ in $\T^q,$ then
there exists a coordinate $i\in\{1,2,\cdots,q\}$ and a maximum-size
independent set $J$ in $\T,$ such that
$$I=\{(v_1,\cdots,v_q\in V(\T^q):v\in J\}.$$
\end{lemma}

\begin{lemma} ${{\fs\cite{Ku2}}}$\label{yl-3.6}
Let $\T$ be a connected, non-bipartite vertex-transitive graph.
Suppose that the only independent sets of maximal cardinality in
$H^2$ are the preimages of the independent sets of maximal
cardinality in $\T$ under projections. Then the same holds for all
powers of $\T.$
\end{lemma}

\begin{proof}  (of Theorem~\ref{dl-1.2}).  For $n=5,6,$ it is easy to see that $A\T_n$
is connected, non-bipartite and $e(n)$-regular graph with $e(5)=24$
and $e(6)=130.$ Moreover, a {\bf Matlab} computation shows that the
least eigenvalue of the adjacency matrix of $A\T_5$ and $A\T_6$ are
$-6$ and $-26,$ respectively. Thus the assertion holds by Lemmas
\ref{yl-3.1} and \ref{yl-3.5}.

For $n\geq 7,$ combining Proposition \ref{pr-1.1}, Lemma
\ref{yl-3.4} and \ref{yl-3.6}, the assertion holds.\end{proof}

\begin{corollary} \label{co-3.7}
Let $\omega(A\T_n^q)$ and $\chi(A\T_n^q)$ denote the clique number
and chromatic number of $A\T_n^q~(n\geq 5).$ Then we have
$$\omega(A\T_n^q)=\chi(A\T_n^q)=n.$$
\end{corollary}

\begin{proof}  By \cite {Ku}, we have $\omega(A\T_n)=n.$ Let
$\{\sigma_1,\sigma_2,\cdots,\sigma_n\}$ be a clique of $A\T_n.$ Then
clearly
$\{(\sigma_1,\sigma_1,\cdots,\sigma_1),(\sigma_2,\sigma_2,\cdots,\sigma_2),\cdots,(\sigma_n,\sigma_n,\cdots,\sigma_n)\}$
is a clique of $A\T_n^q.$ So we have $\omega(A\T_n^q)\geq n.$ On the
other hand, by Theorem \ref{dl-1.2}, we know that the independence
number $\alpha(A\T_n^q)=\frac{(n-1)!n!^{q-1}}{2^q}.$ By Corollary 4
in \cite {Cameron}, we have $\omega(A\T_n^q)\alpha(A\T_n^q)\leq
|V(A\T_n^q)|,$ that is
$\omega(A\T_n^q)\cdot\frac{(n-1)!n!^{q-1}}{2^q}\leq
\frac{n!^q}{2^q},$ so $\omega(A\T_n^q)\leq n.$ Thus
$\omega(A\T_n^q)=n.$

In addition, by Corollary 6.1.3 in \cite {Godsil}, for any Cayley
graph $\T:=\T(G,S),$ if $S$ is closed under conjugation and
$\alpha(\T)\omega(\T)=|V(\T)|,$ then $\chi(\T)=\omega(\T).$ Note
that for $A\T_n^q=\T(A_n^q,{\mathcal E}_n^q),\,{\mathcal E}_n^q$ is
closed under conjugation and
$\alpha(A\T_n^q)\omega(A\T_n^q)=|V(A\T_n^q)|.$ Hence
$\chi(A\T_n^q)=\omega(A\T_n^q)=n.$
\end{proof}

\section{The automorphism group of $A\T_n^q$}

In this section, we completely determine the full automorphism group
of $A\T_n^q~(n\geq 5).$ First we introduce some definitions. Let
$\Sym(\Omega)$ denote the set of all permutations of a set $\Omega.$
A {\it permutation representation} of a group \G is a homomorphism
from \G into $\Sym(\Omega)$ for some set $\Omega.$ A permutation
representation is also referred to as an action of \G on the set
$\Omega,$ in which case we say that \G acts on $\Omega.$
Furthermore, if $\{g\in G:x^g=x,\,\forall x\in \Omega\}=1,$ we say
the action of $G$ on $\Omega$ is {\it faithful}, or $G$ acts {\it
faithfully} on $\Omega.$

Next we need the following known results:

\begin{proposition} ${{\fs\cite{Johnson}}}$ \label{pr-4.1}
Let $G^q=G\times G\times\cdots\times G$ be the external direct
product of $q$ copies of the nontrivial group $G.$ If $G$ has the
following properties:

(i) the center $Z(G)$ of $G$ is trivial;

(ii) G cannot be decomposed as a nontrivial direct product.\\
Then $\Aut(G^q)=\Aut(G)\wr S_q.$
\end{proposition}

\begin{proposition} ${{\fs\cite{Godsil3}}}$ \label{pr-4.2}
Let $N_{Aut(\T(G,S)}(R(G))$ be the normalizer of $R(G)$ in
$\Aut(\T(G,S)).$ Then
$$N_{Aut(\T(G,S)}(R(G))=R(G)\rtimes \Aut(G,S)\leq \Aut(\T(G,S)),$$
where $\Aut(G,S)=\{\phi\in \Aut(G):\,S^{\phi}=S\}.$
\end{proposition}

\begin{lemma} \label{yl-4.3}
Define the mapping $\varphi_k:\,A_n^q\rightarrow A_n^q$ as
$(\sigma_1,\cdots,\sigma_{k-1},\sigma_k,\sigma_{k+1},\cdots,\sigma_q)^{\varphi_k}=(\sigma_1,\cdots,\sigma_{k-1},\sigma_k^{-1},
\sigma_{k+1},\cdots,\sigma_q)$ for $k=1,2,\cdots,q.$ For $n\geq 5,$
$$(R(A_n^q)\rtimes
(\Inn(S_n)\wr S_q))\rtimes Z_2^q\leq \Aut(A\T_n^q), $$ where
$\Inn(S_n)~\cong S_n$ and
$Z_2^q=\langle\varphi_1\rangle\times\langle\varphi_2\rangle\times\cdots\langle\varphi_q\rangle.$
In particular, $|\Aut(A\T_n^q)|\geq |(R(A_n^q)\rtimes (\Inn(S_n)\wr
S_q))\rtimes Z_2^q|=q!n!^{2q}.$
\end{lemma}

\begin{proof}
By Proposition~\ref {pr-2.1} and \ref {pr-4.1}, we have
\begin{eqnarray*}
\Aut(A_n^q,{\mathcal E}_n^q)&=&\{\phi\in \Aut(A_n^q):\,({\mathcal E}_n^q)^{\,\phi}={\mathcal E}_n^q\}\\
&=&\{\phi\in \Aut(A_n)\wr Sq:\,({\mathcal E}_n^q)^{\,\phi}={\mathcal E}_n^q\}\\
&=&\Inn(S_n)\wr Sq.
\end{eqnarray*}

Using Proposition \ref {pr-4.2}, we obtain $R(A_n^q)\rtimes
(\Inn(S_n)\wr S_q)\leq \Aut(A\T_n^q).$

Next we show that $\varphi_k$ is an automorphism of $A\T_n^q.$
\begin{eqnarray*}
&&\{(\sigma_1,\cdots,\sigma_k,\cdots,\sigma_q),(\tau_1,\cdots,\tau_k,\cdots,\tau_q)\}\in E(A\T_n^q)\\
&\Leftrightarrow&\forall~i\in\{1,2,\cdots,n\},\forall~k\in\{1,2,\cdots,q\},\,i^{\sigma_k}\neq i^{\tau_k}\\
&\Leftrightarrow&\forall~i\in\{1,2,\cdots,n\},\forall~k\in\{1,2,\cdots,q\},\,(i^{{\sigma_k}^{-1}})^{\sigma_k}\neq (i^{{\sigma_k}^{-1}})^{\tau_k}\\
&\Leftrightarrow&\forall~i\in\{1,2,\cdots,n\},\forall~k\in\{1,2,\cdots,q\},\,i\neq i^{{\sigma_k}^{-1}\tau_k}\\
&\Leftrightarrow&\forall~i\in\{1,2,\cdots,n\},\forall~k\in\{1,2,\cdots,q\},\,i^{{\tau_k}^{-1}}\neq i^{{\sigma_k}^{-1}}\\
&\Leftrightarrow&\{(\sigma_1,\cdots,\sigma_k^{-1},\cdots,\sigma_q),(\tau_1,\cdots,\tau_k^{-1},\cdots,\tau_q)\}\in E(A\T_n^q)\\
&\Leftrightarrow&\{(\sigma_1,\cdots,\sigma_k,\cdots,\sigma_q)^{\varphi_k},(\tau_1,\cdots,\tau_k,\cdots,\tau_q)^{\varphi_k}\}\in
E(A\T_n^q).
\end{eqnarray*}
 It is easy to see that $\varphi_k\not\in R(A_n^q)$ and
$\varphi_k\not\in \Inn(S_n)\wr S_q.$ Hence $$(R(A_n^q)\rtimes
(\Inn(S_n)\wr S_q))\rtimes Z_2^q\leq \Aut(A\T_n^q), $$ where
$Z_2^q=\langle\varphi_1\rangle\times\langle\varphi_2\rangle\times\cdots\langle\varphi_q\rangle.$
The assertion holds.
\end{proof}

\begin{lemma}\label{yl-4.4}
Let ${\mathcal
B}=\{B_{i,j}^{(k)},\,i,j=1,2,\cdots,n;\,k=1,2\cdots,q\}$, where
$B_{i,j}^{(k)}=\{(\sigma_1,\sigma_2,\cdots,\sigma_q)\in
A_n^q:\,i^{\sigma_k}=j\}.$ Then the action of $\Aut(A\T_n^q)$ on
${\mathcal B}$ can be induced by the natural action of
$\Aut(A\T_n^q)$ on $A_n^q$, and is faithful. Furthermore, any
$\phi\in \Aut(A\T_n^q)$ is a permutation of ${\mathcal B}.$
\end{lemma}

\begin{proof}
Obviously, any $\phi\in\Aut(A\T_n^q)$ maps a maximum-size
independent set of $A\T_n^q$ to a maximum-size independent set of
$A\T_n^q.$ So by Theorem \ref{dl-1.2}, for any $B_{i,j}^{(k)}\in
{\mathcal B}$ and $\phi\in\Aut(\T_n^q),$ we have
${B_{i,j}^{(k)}}^{\,\phi}\in {\mathcal B}.$

Next we show that if $\phi\in\Aut(A\T_n^q)$ satisfies
${B_{i,j}^{(k)}}^{\,\phi}=B_{i,j}^{(k)}$ for each $B_{i,j}^{(k)}\in
{\mathcal B},$ then $\phi$ is the identity map. In fact, clearly,
$$\forall\,(\sigma_1,\sigma_2,\cdots,\sigma_q)\in A_n^q,\,\{(\sigma_1,\sigma_2,\cdots,\sigma_q)\}=\bigcap_{k=1}^q\bigcap_{i=1}^n
B_{i,i^{\,\sigma_k}}^{(k)}.$$
So
\begin{eqnarray*} \{(\sigma_1,\sigma_2,\cdots,\sigma_q)^{\,\phi}\}&=&(\bigcap_{k=1}^q\bigcap_{i=1}^n B_{i,i^{\,\sigma_k}}^{(k)})^{\,\phi}\\
&\subseteq& \bigcap_{k=1}^q\bigcap_{i=1}^n {B_{i,i^{\,\sigma_k}}^{(k)}}^{\phi}\\
&=&\bigcap_{k=1}^q\bigcap_{i=1}^nB_{i,i^{\,\sigma_k}}^{(k)}\\
&=&\{(\sigma_1,\sigma_2,\cdots,\sigma_q)\}.
\end{eqnarray*}
Thus $\phi$ is the identity map.

For any $B_{i,j}^{(k)},\,B_{i^{'},j^{'}}^{(k^{'})}\in {\mathcal B}$
and $\phi\in \Aut(A\T_n^q),$ we have
\begin{eqnarray*}
B_{i,j}^{(k)}\neq B_{i^{'},j^{'}}^{(k^{'})}&\Leftrightarrow&|B_{i,j}^{(k)}\cup B_{i^{'},j^{'}}^{(k^{'})}|>\frac{(n-1)!n!^{q-1}}{2^q}\\
&\Leftrightarrow&|(B_{i,j}^{(k)}\cup B_{i^{'},j^{'}}^{(k^{'})})^{\phi}|>\frac{(n-1)!n!^{q-1}}{2^q}\\
&\Leftrightarrow&|{B_{i,j}^{(k)}}^{\,\phi}\cup {B_{i^{'},j^{'}}^{(k^{'})}}^{\phi}|>\frac{(n-1)!n!^{q-1}}{2^q}\\
&\Leftrightarrow&{B_{i,j}^{(k)}}^{\,\phi}\neq
{B_{i^{'},j^{'}}^{(k^{'})}}^{\phi}.
\end{eqnarray*}
Thus $\phi$ is a permutation of ${\mathcal B}.$
\end{proof}

\begin{lemma}\label{yl-4.5}
$B_{i,j}^{(k)}\cap B_{i^{'},j^{'}}^{(k^{'})}=\emptyset$ if and only
if $k=k^{'}$ and exactly one of $i=i^{'}$ and $j=j^{'}$ holds.
\end{lemma}

\begin{proof}
If $k=k^{'}$ and exactly one of $i=i^{'}$ and $j=j^{'}$ holds, then
$B_{i,j}^{(k)}\cap B_{i^{'},j^{'}}^{(k^{'})}=\emptyset.$

If $k\neq k^{'},$ then $|B_{i,j}^{(k)}\cap
B_{i^{'},j^{'}}^{(k^{'})}|=|\{(\sigma_1,\sigma_2,\cdots,\sigma_q)\in
A_n^q:\,i^{\sigma_k}=j,\,{i^{'}}^{\sigma_{k^{'}}}=j^{'}\}|=\frac{(n-1)!^2n!^{q-2}}{2^q}.$

If $k=k^{'},\,i= i^{'}$ and $j= j^{'},$ then $B_{i,j}^{(k)}=
B_{i^{'},j^{'}}^{(k^{'})},$ so $B_{i,j}^{(k)}\cap
B_{i^{'},j^{'}}^{(k^{'})}\neq\emptyset.$

If $k=k^{'},\,i\neq i^{'}$ and $j\neq j^{'},$ then
$$|B_{i,j}^{(k)}\cap
B_{i^{'},j^{'}}^{(k^{'})}|=\{(\sigma_1,\sigma_2,\cdots,\sigma_q)\in
A_n^q:\,i^{\sigma_k}=j,\,{i^{'}}^{\sigma_k}=j^{'}\}|=\frac{(n-1)!n!^{q-1}}{2^q}.$$

Thus the assertion holds.
\end{proof}

\begin{lemma}\label{yl-4.6}
Let ${\mathcal
B}^{(k)}=\{B_{i,j}^{(k)},\,i,j=1,2,\cdots,n\},\,k=1,2,\cdots,q.$ For
any $\phi\in\Aut(A\T_n^q),$ There exists a $B_{i,j}^{(k)}\in
{\mathcal B}^{(k)}$ such that ${B_{i,j}^{(k)}}^{\,\phi}\in {\mathcal
B}^{(k^{'})}$ if and only if ${B_{i,j}^{(k)}}^{\,\phi}\in {\mathcal
B}^{(k^{'})}$ for any $B_{i,j}^{(k)}\in {\mathcal B}^{(k)}.$
\end{lemma}

\begin{proof}
Suppose on the contrary that there exist two distinct
$B_{i,j}^{(k)},\,B_{i^{'},j^{'}}^{(k)}\in {\mathcal B}^{(k)}$ such
that ${B_{i,j}^{(k)}}^{\,\phi}\in {\mathcal
B}^{(k^{'})},\,{B_{i^{'},j^{'}}^{(k)}}^{\phi}\in {\mathcal
B}^{(k^{''})}$ with $k^{'}\neq k^{''}.$

Since $B_{i,j}^{(k)}\neq B_{i^{'},j^{'}}^{(k)},$ we have
$|B_{i,j}^{(k)}\cap
B_{i^{'},j^{'}}^{(k)}|=0~or~\frac{(n-2)!n!^{q-1}}{2^q}$ by using
Lemma \ref{yl-4.5} and its proof. So
\begin{eqnarray*}
|{B_{i,j}^{(k)}}^{\,\phi}\cup
{B_{i^{'},j^{'}}^{(k)}}^{\phi}|&=&|(B_{i,j}^{(k)}\cup
B_{i^{'},j^{'}}^{(k)})^{\,\phi}|=|B_{i,j}^{(k)}\cup
B_{i^{'},j^{'}}^{(k)}|\\
&=&\frac{2(n-1)!n!^{q-1}}{2^q}~or~\frac{2(n-1)!n!^{q-1}-(n-2)!n!^{q-1}}{2^q}.
\end{eqnarray*}
On the other hand,
\begin{eqnarray*}
{B_{i,j}^{(k)}}^{\,\phi}\in
B^{(k^{'})},{B_{i^{'},j^{'}}^{(k)}}^{\phi}\in B^{(k^{''})}\,
(k^{'}\neq k^{''})&\Rightarrow& |{B_{i,j}^{(k)}}^{\,\phi}\cap
{B_{i^{'},j^{'}}^{(k)}}^{\phi}|=\frac{(n-1)!^2n!^{q-2}}{2^q}\\
&\Rightarrow& |{B_{i,j}^{(k)}}^{\,\phi}\cup
{B_{i^{'},j^{'}}^{(k)}}^{\phi}|=\frac{2(n-1)!n!^{q-1}-(n-1)!^2n!^{q-2}}{2^q},
\end{eqnarray*}
 which is a
contradiction. Thus the assertion holds.
\end{proof}

\begin{lemma}\label{yl-4.7}
Let ${\mathcal
R}_i^{(k)}=\{B_{i,1}^{(k)},B_{i,2}^{(k)},\cdots,B_{i,n}^{(k)}\}$ and
${\mathcal
C}_j^{(k)}=\{B_{1,j}^{(k)},B_{2,j}^{(k)},\cdots,B_{n,j}^{(k)}\},\,k=1,2,\cdots,q.$
Then for any $x_1,x_2,\cdots,x_n\in {\mathcal B},$ we have
$$x_1\cup x_2\cup\cdots\cup x_n=A_n^q$$ if and only if there exist some
$k\in\{1,2,\cdots,q\}$ and some $i$ or $j\in \{1,2,\cdots,n\}$ such
that $\{x_1,x_2,\cdots,x_n\}={\mathcal R}_i^{(k)}$ or ${\mathcal
C}_j^{(k)}.$
\end{lemma}

\begin{proof}
Clearly if $\{x_1,x_2,\cdots,x_n\}={\mathcal R}_i^{(k)}$ or
${\mathcal C}_j^{(k)}$ for some $k\in\{1,2,\cdots,q\}$ and some $i$
or $j\in \{1,2,\cdots,n\},$ then $x_1\cup x_2\cup\cdots\cup
x_n=A_n^q.$

Assume that $x_1\cup x_2\cup\cdots\cup x_n=A_n^q.$ Since
$\forall\,i,\,|x_i|=\frac{(n-1)!n!^{q-1}}{2^q}$ and
$|A_n^q|=\frac{n!^q}{2^q},$  we have $x_i\cap
x_j=\emptyset,\,\forall i,j,\,i\neq j.$ Applying Lemma \ref{yl-4.5},
we obtain $\{x_1,x_2,\cdots,x_n\}={\mathcal R}_i^{(k)}$ or
${\mathcal C}_j^{(k)}.$
\end{proof}

\begin{lemma}\label{yl-4.8}
Let $\Omega=\{{\mathcal C}_i^{(k)},{\mathcal
R}_j^{(k)},i,j=1,2,\cdots,n;\,k=1,2,\cdots,q\}.$ Then the action of
$\Aut(A\T_n^q)$ on $\Omega$ can be induced by the action of
$\Aut(A\T_n^q)$ on ${\mathcal B}$ in Lemma \ref{yl-4.4}, and it is
faithful. Furthermore, any $\phi\in \Aut(A\T_n^q)$ is a permutation
of $\Omega.$
\end{lemma}

\begin{proof}
First for any ${\mathcal R}_i^{(k)}\in \Omega$ and
$\phi\in\Aut(A\T_n^q),$ we have
$${B_{i,1}^{(k)}}^{\,\phi}\cup{B_{i,2}^{(k)}}^{\,\phi}\cup\cdots\cup{B_{i,n}^{(k)}}^{\,\phi}=
(B_{i,1}^{(k)}\cup B_{i,2}^{(k)}\cup\cdots\cup
B_{i,n}^{(k)})^{\,\phi}=(A_n^q)^{\,\phi}=A_n^q.$$ So by Lemma
\ref{yl-4.7}, we have ${{\mathcal
R}_i^{(k)}}^{\,\phi}=\{{B_{i,1}^{(k)}}^{\,\phi},{B_{i,2}^{(k)}}^{\,\phi},\cdots,{B_{i,n}^{(k)}}^{\,\phi}\}\in\Omega.$

Similarly, for any ${\mathcal C}_j^{(k)}\in \Omega$ and
$\phi\in\Aut(A\T_n^q),$ we have ${{\mathcal C}_j^{(k)}}^{\,\phi}\in
\Omega.$

Assume that $\phi\in \Aut(A\T_n^q)$ satisfies ${{\mathcal
R}_i^{(k)}}^{\,\phi}={\mathcal R}_i^{(k)}$ and ${{\mathcal
C}_j^{(k)}}^{\,\phi}={\mathcal C}_j^{(k)}$ for any $i,j\in
\{1,2,\cdots,n\}$ and $k\in \{1,2,\cdots,q\}.$  Then it suffices to
show that $\phi$ is the identity map.

Since for any $B_{i,j}^{(k)}\in {\mathcal B},$ we have
$\{{B_{i,j}^{(k)}}^{\,\phi}\}=({\mathcal R}_i^{(k)}\cap {\mathcal
C}_j^{(k)})^{\,\phi} \subseteq {{\mathcal R}_i^{(k)}}^{\,\phi}\cap
{{\mathcal C}_j^{(k)}}^{\,\phi} ={\mathcal R}_i^{(k)}\cap {\mathcal
C}_j^{(k)}=\{B_{i,j}^{(k)}\}.$ By Lemma \ref{yl-4.4}, the action of
$\Aut(A\T_n^q)$ on ${\mathcal B}$ is faithful. Thus $\phi$ is the
identity map.

For any $\omega_1,\omega_2\in \Omega$ and $\phi\in \Aut(A\T_n^q),$
\begin{eqnarray*}
\omega_1\neq\omega_2&\Leftrightarrow&|\omega_1\cup\omega_2|>n\\
&\Leftrightarrow&|(\omega_1\cup\omega_2)^{\phi}|>n\\
&\Leftrightarrow&|\omega_1^{\phi}\cup\omega_2^{\phi}|>n\\
&\Leftrightarrow&\omega_1^{\phi}\neq\omega_2^{\phi}.
\end{eqnarray*}
Thus $\phi$ is a permutation of $\Omega.$
\end{proof}

\begin{lemma}\label{yl-4.9}
Let ${\mathcal R}^{(k)}=\{{\mathcal R}_1^{(k)},{\mathcal
R}_2^{(k)},\cdots,{\mathcal R}_n^{(k)}\},\,{\mathcal
C}^{(k)}=\{{\mathcal C}_1^{(k)},{\mathcal
C}_2^{(k)},\cdots,{\mathcal C}_n^{(k)}\}$ and
$\Omega^{(k)}={\mathcal R}^{(k)}\cup {\mathcal
C}^{(k)},\,k=1,2,\cdots,q.$ For any $\phi\in \Aut(A\T_n^q),$ the
following (i)-(iii) hold:

(i) There exists a $\sigma\in S_q$ such that
${\Omega^{(k)}}^{\,\phi}=\Omega^{(k^{\sigma})},\,k=1,2,\cdots,q.$

(ii) There exists some ${\mathcal R}_i^{(k)}\in {\mathcal R}^{(k)}$
such that ${{\mathcal R}_i^{(k)}}^{\,\phi}\in {\mathcal
R}^{(k^{'})}$ if and only if ${{\mathcal R}_i^{(k)}}^{\,\phi}\in
{\mathcal R}^{(k^{'})}$ for any ${\mathcal R}_i^{(k)}\in {\mathcal
R}^{(k)};$

(iii) There exists some ${\mathcal R}_j^{(k)}\in {\mathcal R}^{(k)}$
such that ${{\mathcal R}_j^{(k)}}^{\,\phi}\in {\mathcal
C}^{(k^{'})}$ if and only if ${{\mathcal R}_j^{(k)}}^{\,\phi}\in
{\mathcal C}^{(k^{'})}$ for any ${\mathcal R}_j^{(k)}\in {\mathcal
R}^{(k)}.$
\end{lemma}

\begin{proof} (i) By Lemma \ref{yl-4.6}, for any $k\in \{1,2,\cdots,q\}$ there
exists a $l\in \{1,2,\cdots,q\}$ such that ${{\mathcal
B}^{(k)}}^{\,\phi}={\mathcal B}^{(l)}.$ Moreover, if $k\neq k^{'},$
then by Lemma \ref{yl-4.4}, we have ${{\mathcal
B}^{(k)}}^{\,\phi}\neq{{\mathcal B}^{(k^{'})}}^{\,\phi}.$ Thus there
exists a $\sigma\in S_q$ such that ${{\mathcal
B}^{(k)}}^{\,\phi}={\mathcal B}^{(k^{\sigma})},\,k=1,2,\cdots,q.$ By
Lemma \ref{yl-4.8}, the assertion holds.

(ii) First by (i), there exists some ${\mathcal R}_i^{(k)}\in
{\mathcal R}^{(k)}$ such that ${{\mathcal R}_i^{(k)}}^{\,\phi}\in
\Omega^{(k^{'})}$ if and only if ${{\mathcal R}_i^{(k)}}^{\,\phi}\in
\Omega^{(k^{'})}$ for any ${\mathcal R}_i^{(k)}\in {\mathcal
R}^{(k)}.$

Suppose on the contrary that there exist $i,j\,(\neq i)\in
\{1,2,\cdots,n\}$ such that ${{\mathcal R}_i^{(k)}}^{\,\phi}\in
{\mathcal R}^{(k^{'})}$ and ${{\mathcal R}_j^{(k)}}^{\,\phi}\in
{\mathcal C}^{(k^{'})}.$

Note that
$${\mathcal
R}_i^{(k)}\cap {\mathcal R}_j^{(k)}=\emptyset~for~i\neq j,$$
$${\mathcal
R}_i^{(k)}\cap {\mathcal C}_j^{(k)}=\{B_{i,j}^{(k)}\}~for~any~i,j.$$
Then
$$i\neq j\Rightarrow {\mathcal
R}_i^{(k)}\cap {\mathcal R}_j^{(k)}=\emptyset \Rightarrow |{\mathcal
R}_i^{(k)}\cup {\mathcal R}_j^{(k)}|=2n\Rightarrow |{{\mathcal
R}_i^{(k)}}^{\,\phi}\cup {{\mathcal
R}_j^{(k)}}^{\,\phi}|=|({\mathcal R}_i^{(k)}\cup {\mathcal
R}_j^{(k)})^{\,\phi}|=2n.$$ On the other hand,
$${{\mathcal
R}_i^{(k)}}^{\,\phi}\in {\mathcal R}^{(k^{'})},\,{{\mathcal
R}_j^{(k)}}^{\,\phi}\in {\mathcal C}^{(k^{'})}\Rightarrow
|{{\mathcal R}_i^{(k)}}^{\,\phi}\cap {{\mathcal
R}_j^{(k)}}^{\,\phi}|=1\Rightarrow |{{\mathcal
R}_i^{(k)}}^{\,\phi}\cup {{\mathcal R}_j^{(k)}}^{\,\phi}|=2n-1,$$
which is a contradiction. Thus the assertion holds.

(iii) The proof of (iii) is similar to that of (ii).
\end{proof}

\begin{lemma} \label{yl-4.10}
For $n\geq 5,$ we have
    $$|\Aut(A\T_n^q)|\leq q!n!^{2q}.$$
\end{lemma}

\begin{proof} By
(i) of Lemma \ref{yl-4.9}, for any $\phi\in \Aut(A\T_n^q),$ there
exists a $\sigma\in S_q$ such that
${\Omega^{(k)}}^{\,\phi}=\Omega^{(k^{\sigma})}~(k=1,2,\cdots,q).$
Using (ii) and (iii) of Lemma \ref{yl-4.9} we obtain the following
disjoint alternatives:

(i) ${{\mathcal R}^{(k)}}^{\,\phi}={\mathcal R}^{(k^{\sigma})}$ and
${{\mathcal C}^{(k)}}^{\,\phi}={\mathcal C}^{(k^{\sigma})};$

(ii) ${{\mathcal R}^{(k)}}^{\,\phi}={\mathcal C}^{(k^{\sigma})}$ and
${{\mathcal C}^{(k)}}^{\,\phi}={\mathcal R}^{(k^{\sigma})}.$

So $\Aut(A\T_n^q)=\bigcup_{\sigma\in S_q}\{\phi\in
\Aut(A\T_n^q):{\Omega^{(k)}}^{\,\phi}=\Omega^{(k^{\sigma})},k=1,2,\cdots,q\}.$
Hence, if we can prove the last two inequalities, then we have
\begin{eqnarray*}
|\Aut(A\T_n^q)|&\leq&\Sigma_{\sigma\in S_q}|\{\phi\in
\Aut(A\T_n^q):{\Omega^{(k)}}^{\,\phi}=\Omega^{(k^{\sigma})},k=1,2,\cdots,q\}|\\
&\leq&\Sigma_{\sigma\in S_q}\Pi_{k=1}^q(|\{\phi\in
\Aut(A\T_n^q):{{\mathcal R}^{(k)}}^{\,\phi}={\mathcal
R}^{(k^{\sigma})},\,{{\mathcal C}^{(k)}}^{\,\phi}={\mathcal
C}^{(k^{\sigma})}\}|+\\
&&~~~~~~~~~~~~~~~~~|\{\phi\in \Aut(A\T_n^q):{{\mathcal
R}^{(k)}}^{\,\phi}={\mathcal C}^{(k^{\sigma})},\,{{\mathcal
C}^{(k)}}^{\,\phi}={\mathcal
R}^{(k^{\sigma})}\}|)\\
&\leq&\Sigma_{\sigma\in S_q}\Pi_{k=1}^q(\frac{n!^2}{2}+\frac{n!^2}{2})\\
&=&\Sigma_{\sigma\in S_q}n!^{2q}\\
&=&q!n!^{2q}.
\end{eqnarray*}

Now we show that
$$|\{\phi\in \Aut(A\T_n^q):{{\mathcal
R}^{(k)}}^{\,\phi}={\mathcal R}^{(k^{\sigma})},\,{{\mathcal
C}^{(k)}}^{\,\phi}={\mathcal C}^{(k^{\sigma})}\}|\leq
\frac{n!^2}{2},$$
$$|\{\phi\in \Aut(A\T_n^q):{{\mathcal
R}^{(k)}}^{\,\phi}=C^{(k^{\sigma})},\,{C^{(k)}}^{\,\phi}={\mathcal
R}^{(k^{\sigma})}\}|\leq \frac{n!^2}{2}.$$

Indeed, for any $\phi\in \Aut(A\T_n^q)$ such that ${{\mathcal
R}^{(k)}}^{\,\phi}={\mathcal R}^{(k^{\sigma})},\,{{\mathcal
C}^{(k)}}^{\,\phi}={\mathcal C}^{(k^{\sigma})},$ define
$\phi_1,\phi_2\in S_n$ as ${{\mathcal R}_i^{(k)}}^{\phi}={\mathcal
R}_{i^{\,\phi_1}}^{(k^{\sigma})},\,{{\mathcal
C}_j^{(k)}}^{\phi}={\mathcal C}_{j^{\,\phi_2}}^{(k^{\sigma})}.$

Since $\{{B_{ij}^{(k)}}^{\phi}\}=({\mathcal R}_i^{(k)}\cap {\mathcal
C}_j^{(k)})^{\phi}\subseteq {{\mathcal
R}_i^{(k)}}^{\phi}\cap{{\mathcal C}_j^{(k)}}^{\phi}={\mathcal
R}_{i^{\,\phi_1}}^{(k^{\sigma})}\cap {\mathcal
C}_{j^{\,\phi_2}}^{(k^{\sigma})}=\{B_{i^{\,\phi_1}j^{\,\phi_2}}^{(k^{\sigma})}\},$
we have
$$\{(1,1,\cdots,1)^{\phi}\}\in (\bigcap_{i=1}^n B_{ii}^{(k)})^{\phi}
\subseteq\bigcap_{i=1}^n {B_{ii}^{(k)}}^{\phi} =\bigcap_{i=1}^n
B_{i^{\,\phi_1}i^{\,\phi_2}}^{(k^{\sigma})}
=\{(\tau_1,\tau_2,\cdots,\tau_q)\in
A_n^q:\tau_{k^{\sigma}}=\phi_1^{-1}\phi_2\}.$$ So
$(1,1,\cdots,1)^{\phi}=(\tau_1,\cdots,\tau_{k^{\sigma}-1},\phi_1^{-1}\phi_2,\tau_{k^{\sigma}+1}\cdots,\tau_q)\in
A_n^q,$ which implies that $\phi_1^{-1}\phi_2\in A_n.$

Thus \begin{eqnarray*} &&|\{\phi\in \Aut(A\T_n^q):{{\mathcal
R}^{(k)}}^{\,\phi}={\mathcal R}^{(k^{\sigma})},\,{{\mathcal
C}^{(k)}}^{\,\phi}={\mathcal
C}^{(k^{\sigma})}\}|\\
&=&|\{\phi\in \Aut(A\T_n^q):{{\mathcal R}^{(k)}}^{\,\phi}={\mathcal
R}^{(k^{\sigma})},\,{{\mathcal C}^{(k)}}^{\,\phi}={\mathcal
C}^{(k^{\sigma})},\,\phi_1^{-1}\phi_2\in
A_n\}|\\
&\leq& \frac{n!^2}{2}.
\end{eqnarray*}

Similarly,
$$|\{\phi\in \Aut(A\T_n^q):{{\mathcal
C}^{(k)}}^{\,\phi}={\mathcal R}^{(k^{\sigma})},\,{{\mathcal
C}^{(k)}}^{\,\phi}={\mathcal R}^{(k^{\sigma})}\}|\leq
\frac{n!^2}{2}.$$

Thus the assertion holds.
\end{proof}

\begin{proof}  (of Theorem~\ref{dl-1.3}).  By Lemma~\ref{yl-4.3},  we have
 $|\Aut(A\T_n^q)|\geq q!n!^{2q}.$  On the other hand,
by Lemma~\ref{yl-4.10}, we obtain $|\Aut(A\T_n^q|\leq q!n!^{2q}.$
Hence $|\Aut(A\T_n^q)|=q!n!^{2q},$ and by Lemma~\ref{yl-4.3} again,
the assertion holds.
\end{proof}


\frenchspacing
\begin{thebibliography}{}

\bibitem{Albertson} Albertson M O, Collins K L. Homomorphisms of $3$-chromatic
graphs. \emph{Discrete Math}, {\bf 54}: 127--132 (1985)

\bibitem{Alon} Alon N, Dinur I, Friedgut E, et al. Graph products,
fourier analysis and spectral techniques. \emph{Geomatric And
Functional Analysis}, {\bf 14}: 913--940 (2004)

\bibitem{Cameron} Cameron P J, Ku C Y. Intersecting families of
permutations. \emph{European J Combin}, {\bf 24}: 881--890 (2003)

\bibitem{deza1977} Deza M, Frank P. On the maximum number of
permutations given maximal or minimal distance. \emph{J Combin
Theory Ser A}, {\bf 22}: 352--360 (1977)

\bibitem{Eggleton1985} Eggleton R B, Wallis W D. Problem 86:
Solution I. \emph{Math Mag}, {\bf 58}: 112--113 (1985)

\bibitem{Fang} Fang X G, Praeger C E, Wang J. On the automorphism groups of Cayley graphs of finite simple
groups. \emph{J London Math Soc}, {\bf (2)66}: 563--578 (2002)

\bibitem{Feng2} Feng Y Q. Automorphism groups of Cayley graphs on symmetric groups
with generating transposition sets. \emph{J Combin Theory Ser B},
{\bf 96}: 67--72 (2006)

\bibitem{Feng1} Feng Y Q, Xu M Y. Automorphism groups of tetravalent Cayley graphs on regular
p-groups. \emph{Discrete Math}, {\bf 305}: 354--360 (2005)

\bibitem{Godsil} Godsil C. Interesting Graphs and Their Colourings.
http://www.math.ntu.edu.tw/$\sim$
gjchang/courses/2008-09-algebraic-graph-theory/Interesting\%20Graphs\%20and\%20
their\%20Colourings\%20(Godsil).pdf

\bibitem{Godsil2} Godsil C D. The automorphism groups of some cubic Cayley
graphs. \emph{European J Combin}, {\bf 4}: 25--32 (1983)

\bibitem{Godsil3} Godsil C D. On the full automorphism group of a
graph. \emph{Combinatorica}, {\bf 1}: 243--256 (1981)

\bibitem{Graham} Graham R L, Gr\"{o}tschel M, Lov\'{a}sz L. Handbook of
Combinatorics v II, Cambridge: The MIT Press (1995)

\bibitem{Huan} Huan H L, Liu H M, Xie W. Automorphism groups of a family of Cayley graphs on alternating
groups. \emph{J Syst Sci Inform}, {\bf 5}: 37--42 (2007)

\bibitem{Johnson} Johnson E A. Automorphism groups of direct products of groups and their geometric realisations.
\emph{Math Ann}, {\bf 263}: 343--364 (1983)

\bibitem{Ku2} Ku C Y, McMillan B. Independent sets of maximal size in tensor powers of vertex-transitive
graphs. \emph{J Graph Theory}, {\bf (4)60}: 295--301 (2009)

\bibitem{Ku}  Ku C Y, Wong T W H. Intersecting families in the alternating
group and direct product of symmetric groups. \emph{Electron J
Combin}, {\bf 14}:\# R25 (2007)

\bibitem{Larose} Larose B, Malvenuto C. Stable sets of maximal size in
Kneser-type graphs. \emph{European J Combin}, {\bf 14}: 657--673
(2004)

\bibitem{Larose2} Larose B, Tardif C. Projectivity and independent sets in powers of
graphs. \emph{J Graph Theory}, {\bf (3)40}: 162--171 (2002)

\bibitem{Paul} Renteln P. On the spectrum of the derangement Graph. \emph{Electron J
Combin}, {\bf 14}: \# R82 (2007)

\bibitem{Rasmussen} Rasmussen D J, Savage C D. Hamilton-connected derangement
graphs on $S_n$. \emph{Discrete Math}, {\bf 133}: 217--223 (1994)

\bibitem{Sanders} Sanders R S, George J C. Basic results concerning the automorphism
group of the tensor product of two graphs. \emph{Utilitas Math},
51--63 (June 1997)

\bibitem{Suzuki} Suzuki M. Group theory I, New York: Springer (1982)

\bibitem{Xu} Xu M Y. Automorphism groups and isomorphisms of Cayley
digraphs.  \emph{Discrete Math}, {\bf 182}: 309--319 (1998)

\bibitem{Z} Zhang Z, Huang Q X. Automorphism group of bubble-sort graphs and
modified bubble-sort graphs. \emph{Adv Math}, {\bf 34}: 441--447
(2005)

\bibitem{Zhang} Zhang C, Zhou J X, Feng Y Q. Automorphisms of cubic Cayley graphs of order
2pq. \emph{Discrete Math}, {\bf 309}: 2687--2695 (2009)


\end{thebibliography}
\end{document}